\documentclass[11pt]{amsart}

\usepackage[utf8]{inputenc}
\usepackage[T1]{fontenc}

\usepackage{enumitem}
\setlist[enumerate]{label=\upshape{(\roman*)}}

\usepackage{mathtools}
\usepackage{amssymb}
\usepackage{amsthm}
\usepackage{graphicx}
\usepackage{tikz}
\usepackage{xcolor}

\allowdisplaybreaks

\usepackage[margin=1.2in]{geometry}
\usepackage{bookmark}

\def\eps{\varepsilon}

\def\cC{\mathcal {C}}

\def\cG{\mathcal {G}}
\def\cP{\mathcal {P}}
\def\cT{\mathcal {T}}

\def\1{\mathbf{1}}

\def\lam {\lambda}

\def\tce{t_c + \eps}
\def\tce2{t_c + \frac{\eps}{2}}

\DeclareMathOperator{\inj}{inj}

\DeclareMathOperator{\sub}{sub}

\newtheorem*{theorem*}{Theorem}
\newtheorem{theorem}{Theorem}
\newtheorem{lemma}[theorem]{Lemma}

\newtheorem{defn}[theorem]{Definition}
\newtheorem*{defn*}{Definition}

\newtheorem*{prop*}{Proposition}
\newtheorem{conj}[theorem]{Conjecture}
\newtheorem*{conj*}{Conjecture}

\newtheorem*{fact*}{Fact}

\usetikzlibrary{shapes.geometric}
\tikzset{gon/.style={name=tmp,regular polygon,regular polygon sides=#1,minimum
size=10pt,inner sep=0pt},
polygon side/.style args={#1--#2}{
insert path={(tmp.corner #1)-- (tmp.corner #2)}}}
\newcommand{\FlagGraph}[3][]{\ifnum#2=2%
\tikz[baseline=(tmp1)]{\node[circle,inner sep=0.7pt,fill] (tmp1) at (0,0){};
\node[#1,circle,inner sep=0.7pt,fill] (tmp2) at (0,10pt){};
\ifx#3\empty%
\else
\draw[#1] (tmp1) -- (tmp2);
\fi}
\else%
\tikz[baseline=(tmp.south)]{\node[#1,gon=#2]{};
\foreach \X in {1,...,#2}{\fill (tmp.corner \X) circle (1pt);}
\draw[#1,polygon side/.list={#3}]}
\fi}

\newcommand*{\Dgraph}{\FlagGraph{4}{1--2,2--3,3--4,4--1,1--3}}

\newcommand*{\Kfourgraph}{\FlagGraph{4}{1--2,2--3,3--4,4--1,1--3,2--4}}

\begin{document}
\title[The Upper Matching Conjecture for large graphs]{A proof of the Upper Matching Conjecture \\for large graphs}

\author[E.\ Davies]{Ewan Davies}
\address{Department of Computer Science, University of Colorado Boulder, Boulder, CO}
\email{maths@ewandavies.org}

\author[M.\ Jenssen]{Matthew Jenssen}
\address{
School of Mathematics, University of Birmingham, Edgbaston, Birmingham, B15 2TT, UK}
\email{m.jenssen@bham.ac.uk}

\author[W.\ Perkins]{Will Perkins}
\address{Department of Mathematics, Statistics, and Computer Science, University of Illinois at Chicago, Chicago, IL}
\email{math@willperkins.org}

\date{\today}

\begin{abstract}
  We prove that the `Upper Matching Conjecture' of Friedland, Krop, and Markstr{\"o}m and the analogous conjecture of Kahn for independent sets in regular graphs hold for all large enough graphs as a function of the degree.   That is, for every $d$ and every large enough $n$ divisible by $2d$, a union of $n/(2d)$ copies of the complete $d$-regular bipartite graph maximizes the number of independent sets and matchings of size $k$ for each $k$ over all $d$-regular graphs on $n$ vertices.  To prove this we utilize the cluster expansion for the canonical ensemble of a statistical physics spin model, and we give some further applications of this method to maximizing and minimizing the number of independent sets and matchings of a given size in regular graphs of a given minimum girth.  
\end{abstract}

\maketitle

\section{Introduction}

Let $i_k(G)$ be the number of independent sets of size $k$ in a graph $G$ and let $m_k(G)$ be the number of matchings of size $k$ (that is, of $k$ edges).  Then we can write the independence polynomial and matching polynomial (or matching generating function) of $G$ as
\begin{align*}
  Z_G(\lam)     = \sum_{k \ge 0} i_k(G) \lam^k \quad \text{and} \quad
  Z_G^{m}(\lam) = \sum_{k \ge 0} m_k(G) \lam^k 
\end{align*}
respectively.   Evaluating $Z_G(\lam)$ and $Z^m_G(\lam)$ at $\lam=1$ gives the total number of independent sets and matchings in $G$ respectively, which we denote by $i(G)$ and $m(G)$.

Extremal properties of $i(G)$ and $m(G)$ as well as $Z_G(\lam)$ and $Z^m_G(\lam)$ over the class of $d$-regular graphs have been studied since Granville asked which $d$-regular graph on $n$ vertices maximizes the number of independent sets and Alon~\cite{alon1991independent} conjectured that (when $n$ is divisible by $2d$) the answer was the graph $H_{d,n}$, the union of $n/(2d)$ copies of the complete $d$-regular bipartite graph $K_{d,d}$.  Using a beautiful entropy argument, Kahn~\cite{kahn2001entropy} proved this conjecture over the class of $d$-regular bipartite graphs. Galvin and Tetali~\cite{galvin2004weighted} later gave a broad generalization of Kahn's result, showing in particular that $K_{d,d}$ maximizes $\frac{1}{|V(G)|} \log Z_G(\lam)$ for all $\lam \ge 0$ over $d$-regular bipartite graphs. Zhao~\cite{zhao2010number}, then removed the bipartite restriction in this statement for the independence polynomial, resolving Alon's conjecture as a result.  In his 2001 paper Kahn conjectured a stronger result, that $H_{d,n}$ should maximize \emph{each} coefficient of the independence polynomial.
\begin{conj}[Kahn~\cite{kahn2001entropy}]
  \label{conjKahn}
  For all $d \ge 1$, all $n$ divisible by $2d$,  all $d$-regular graphs $G$ on $n$ vertices, and  all $k$, we have   
  \begin{align*}
    i_k(G) & \le i_k(H_{d,n})\,.
  \end{align*}
\end{conj}

The history of the problem for matchings is somewhat different but ends in roughly the same place.  For an $n$ vertex graph $G$, $m_{n/2}(G)$ counts the number of perfect matchings of $G$.  Bregman's Theorem~\cite{bregman1973some} (and its extension from bipartite graphs to general graphs by Kahn and Lov{\'a}sz) gives an upper bound on $m_{n/2}(G)$ in terms of the degree sequence of $G$.  In the case where $G$ is $d$-regular and $n$ is divisible by $2d$, the result states that $m_{n/2}(G)$ is maximized by $H_{d,n}$.  Both Kahn's result on independent sets and Bregman's theorem have elegant proofs using the entropy method~\cite{radhakrishnan1997entropy,galvin2014three}.

In 2008, Friedland, Krop, and Markstr\"{o}m~\cite{friedland2008number} made the equivalent of Conjecture~\ref{conjKahn} for matchings.
\begin{conj}[Friedland, Krop, Markstr\"{o}m~\cite{friedland2008number}]
  \label{conjUM}
   For all $d \ge 1$, all $n$ divisible by $2d$,  all $d$-regular graphs $G$ on $n$ vertices, and  all $k$, we have  
  \begin{align*}
    m_k(G) & \le m_k(H_{d,n})\,.
  \end{align*}
\end{conj}
The authors of \cite{friedland2008number} named this the `Upper Matching Conjecture'.
A weaker form of this conjecture was simply that $H_{d,n}$ maximizes $m(G)$ or $Z_G^m(\lam)$ over $d$-regular graphs, but unlike the case of independent sets no entropy-based proof of this is known.  In light of this (and the appealing name `Upper Matching Conjecture'), somewhat more attention was paid to Conjecture~\ref{conjUM} than Conjecture~\ref{conjKahn}.  Partial progress was made on both conjectures in a series of papers including~\cite{friedland2008validations,carroll2009matchings,ilinca2013asymptotics,perkins2015birthday}, though the upper bounds on $m_k(G)$ remained an exponential factor larger than $m_k(H_{d,n})$.

In 2017, the current authors and Roberts introduced a new method for proving extremal bounds on graph polynomials such as $Z_G$ and $Z_G^m$ based on logarithmic derivatives and linear programming relaxations~\cite{davies2015independentns}.  One of the results proved with this method was that $K_{d,d}$ maximizes $\frac{1}{|V(G)|} \log Z_G^m(\lam)$ for all $\lam$ over $d$-regular graphs, and as a consequence, $H_{d,n}$ maximizes $m(G)$.

In addition, \cite{davies2015independentns} provided a new approach to Conjectures~\ref{conjKahn} and~\ref{conjUM}, by developing generic methods for transferring bounds on graph polynomials to bounds on their individual coefficients.   This approach gave upper bounds on $i_k(G)$ and $m_k(G)$ within a factor $O(\sqrt n)$ of the conjectured bounds.
Then in~\cite{DaviesCoefficients}, the same authors gave a more sophisticated version of this approach using stability-type results for $Z_G$ and $Z_G^m$ to show that Conjectures~\ref{conjKahn} and~\ref{conjUM} hold when $n$ is large and $k\ge \eps n$ (see Theorem~\ref{thmStability} below for the precise statement).
In this paper, we use a new method to deal with small values of $k$ (that is, $k < \eps n)$.  Combined with the results from~\cite{DaviesCoefficients} this allows us to prove our main result, that Conjectures~\ref{conjKahn} and~\ref{conjUM} hold for $n\geq n(d)$ large enough.

\begin{theorem}
  \label{thmMain}
  For all $d \ge 2$ there is $N$ large enough so that for all $n \ge N$  divisible by $2d$,  all $d$-regular graphs $G$ on $n$ vertices, and all $k$,
  \begin{align*}
    i_k(G) & \le i_k(H_{d,n}) \quad \text{and} \quad
    m_k(G)  \le m_k(H_{d,n}) \,.
  \end{align*}
    Moreover, if $G$ is not isomorphic to $H_{d,n}$ then the inequalities are strict for all $4 \le k \le n/2$.
\end{theorem}

We note that for a $d$-regular graph $G$ on $n$ vertices, $i_k(G)=0$ for $k>n/2$, and $i_k(G)=i_k(H_{d,n})$ for $k\leq 3$ if $G$ is triangle free (similarly for $m_k(G)$). Thus the range $4 \le k \le n/2$ in Theorem~\ref{thmMain} is best possible.

\subsection{New techniques}

To address the case $k < \eps n$  we turn to a technique from statistical physics, the cluster expansion, which was originally developed to study the phase diagrams of gases and spin systems~\cite{mayer1940statistical} but which has also found a number of applications in combinatorics and graph theory~\cite{scott2005repulsive,sokal2001bounds,borgs2006absence,borgs2013left,jenssen2019revisited,balogh2020independent}.

In its basic form, the cluster expansion is an infinite series that formally represents the logarithm of a function like $Z_G(\lam)$.  To make practical use of the cluster expansion, one must know for what values of $\lam$ this series converges, and general results tell us that $\lam \le 1/(ed)$ suffices for graphs of maximum degree $d$.  Moreover, when $\lam$ is small enough (as a function of $d$ but not of $n$) the dominant terms of the cluster expansion come from the counts of small subgraphs (edges, triangles, four-cycles, etc.).  This holds more generally; instead of considering the independence polynomial, one can consider a vertex-weighted homomorphism-counting polynomial $Z^H_G$ for any graph $H$. Borgs, Chayes, Kahn, and Lovasz~\cite{borgs2013left} used the cluster expansion to show that for small weights (the equivalent of $\lam \le 1/(ed)$), $\frac{1}{n}\log Z^H_{G_n}$ converges if the sequence $G_n$ of graphs on $n$ vertices is Benjamini-Schramm convergent; that is, its small subgraph densities converge.

Conjectures~\ref{conjKahn} and~\ref{conjUM}, however, deal with the individual coefficients of $Z_G$ and $Z^m_G$ instead of the values of the polynomials themselves. Fortunately there is a natural statistical physics perspective to this as well.  Polynomials such as $Z_G$ and $Z_G^m$ are the partition functions of statistical mechanics models, the hard-core model and the monomer-dimer model respectively.  These are probability distributions over the independent sets and matchings of a graph $G$, in which e.g.\ each independent set has probability $\lam^{|I|}/Z_G(\lam)$. This distribution is known as the \emph{grand canonical ensemble} and represents particles of a gas (occupied vertices of the independent set) in a volume in thermal equilibrium within a much larger volume (and so particles can enter and leave the small volume).  There is another natural probability distribution over independent sets: the uniform distribution over independent sets of size $k$, which represents particles of a gas in a confined volume (the particles cannot escape so their number $k$ is fixed).  This is the \emph{canonical ensemble}, and its partition function is simply $i_k(G)$.   Much of the intuition behind the methods of~\cite{DaviesCoefficients} came from comparing the probabilistic behavior of the grand canonical and canonical ensembles.            

There is also a cluster expansion for the canonical ensemble, first presented by Pulvirenti and Tsagkarogiannis in~\cite{pulvirenti2012cluster} in the setting of Gibbs point processes.  In our setting this gives an infinite series representation of $\log i_k(G)$, and for small $k$ again the dominant terms of its cluster expansion are counts of small subgraphs.  We will use two very simple facts about $K_{d,d}$ to show that it maximizes $i_k$: it has no triangles and it has the highest $4$-cycle density of any $d$-regular graph.
The method is much more general and can be applied to independent sets, matchings, and graph homomorphisms over various classes of graphs.

In a little more detail, the following informal meta-theorem follows naturally from the cluster expansion arguments we develop here.  Let $\cG$ be a class of graphs (e.g. $d$-regular graphs, $d$-regular bipartite graphs, line graphs of $d$-regular graphs, etc.), and let $H \in \cG$ be uniquely optimal for maximizing independent sets of size $j_0$; that is, for $j<j_0$ a union of copies of $H$ has at least as many independent sets of size $j$ as any other graph $G \in \cG$ on the same number of vertices, and for $j = j_0$ it strictly maximizes $i_j(G)$.   Then a union of copies of $H$ is the maximizer of $i_k$ for all $ k \le \eps n$ and $H$ is the maximizer of $\frac{1}{n} \log Z_G(\lam)$ for all $\lam \le \eps$.  See Theorem~\ref{thmGirth} below for one example of such a result, and Theorem~\ref{thmGirthMin} for a similar minimization result.  

Concretely, we can consider the problem of minimizing the number of matchings of a given size in a regular graph (the  minimization problem for independent sets is rather straightforward, see~\cite{CR14} and the discussion in~\cite{DaviesCoefficients}).   For even $d$, a natural conjecture is that $K_{d+1}$  minimizes $\frac{1}{|V(G)|}\log Z^m_G(\lam)$ for all $\lam$ and a union of copies of $K_{d+1}$ minimizes each coefficient of $Z_G^m$.  For odd $d$, the clique $K_{d+1}$ contains a perfect matching so cannot minimize $\frac{1}{|V(G)|}\log Z^m_G(\lam)$ over $d$-regular graphs for large $\lam$.   Csikv{\'a}ri and others have conjectured that it is the minimizer for $\lam \le 1$.
Similarly, an $n$-vertex union of these cliques cannot minimize the high order coefficients of the matching polynomial when an $n$-vertex $d$-regular graph with no perfect matching exists. 
We make partial progress on these minimization problems for matchings with the following pair of results.

\begin{theorem}\label{thmCliquemin}
There is a constant $c>0$ such that for all $d\ge 2$ and $0<\lam<cd^{-4} $, we have 
\begin{equation}
\label{eqKdMin}
 \frac{1}{|V(G)|}\log Z^m_G(\lam) \ge \frac{1}{d+1}\log Z^m_{K_{d+1}}(\lam) 
 \end{equation}
for all $d$-regular graphs $G$, with strict inequality if $G$ is not a disjoint union of copies of $K_{d+1}$. 

For all $d\ge 2$ there is a constant $\xi = \xi(d)>0$ such that for all $n$ divisible by $d+1$, a disjoint union of $n/(d+1)$ copies of $K_{d+1}$ minimizes $m_k(G)$ for all $k\le \xi n$. 
\end{theorem}

Borb{\'e}nyi and Csikv{\'a}ri~\cite{csikvari20} have independently proved the first statement of Theorem~\ref{thmCliquemin} for a wider range of  $\lam$.

\subsection{Related work}
The problems of maximizing and minimizing the number of independent sets and matchings in regular graphs have many extensions and generalizations, from  asking which regular graphs maximize or minimize the number of homomorphisms to a given fixed graph $H$ (e.g. the number proper $q$-colorings with $H = K_q$) .  Results and open problems in the area are given in Zhao's survey~\cite{zhao2017extremal}.  Further extensions consider irregular graphs with bounds that depend on the degree sequence~\cite{sah2019number}.  Often in these problems the extremal graph is $K_{d,d}$, $K_{d+1}$, or, in a limiting sense, the infinite $d$-regular tree~\cite{csikvari2016extremal}.  This is not always the case, however, as Sernau gave counterexamples to some previous conjectures~\cite{sernau2018graph}.  Nevertheless, for maximizing the number of homomorphisms to a fixed graph over the class of $d$-regular, triangle-free graphs, $K_{d,d}$ is always optimal~\cite{sah2020reverse}.  One can also consider similar questions for hypergraphs, where the above problems are more difficult and remain open in general.  Some bounds and conjectures for independent sets in hypergraphs are given in~\cite{ordentlich2000two,cohen2020number,balogh2020counting}.  One takeaway from the methods of the current paper is that understanding the maximizer and minimizer of small cycle counts in the class of (hyper)graphs considered can give a good indication of the plausibility of a conjectured extremal result.

\subsection{Organization}

In Section~\ref{secProofOutline} we outline the proof of Theorem~\ref{thmMain}, reducing it to a more general statement about independent sets and matchings in regular graphs of a given minimum girth.  In Section~\ref{secCluster} we present the canonical ensemble cluster expansion and give sufficient conditions for its convergence along with tail bounds based on the Koteck\'{y}--Preiss condition~\cite{kotecky1986cluster}.   In Section~\ref{secGirth} we use this to complete the proof of Theorem~\ref{thmMain} for independent sets, and in Section~\ref{secMatchings} we extend this to matchings.   In Section~\ref{secConclude} we prove Theorem~\ref{thmCliquemin} and discuss possible extensions of these results.

\section{Proof outline for Theorem~\ref{thmMain}}
\label{secProofOutline}

The main new technique we introduce in this paper is a way to use the canonical ensemble cluster expansion to bound $i_k(G)$ (and $m_k(G)$) for small values of $k$, that is, for $k \le \eps n$ where $\eps$ depends only on $d$. 
We do this in considerably more generality since the proof is the same and the more general statement gives more intuition. 

The \emph{girth} of a graph is the length of its shortest cycle, and a \emph{homomorphism} from $F$ to $G$ is a map $\phi:V(F)\to V(G)$ such that $ij\in E(F)\Longrightarrow \phi(i)\phi(j)\in E(G)$. 
We write $\inj(F,G)$ for the number of injective homomorphisms from $F$ to $G$. 

\begin{defn}
For $d,g\geq 2$, we call a graph $H$ strictly $(d,g)$-optimal if it is $d$-regular of girth at least $g$, and there exists $\eta=\eta(d,g)>0$ such that
\[ \frac{1}{v(G)}\inj(C_g, G) \le \frac{1}{v(H)}\inj(C_g, H) - \eta \]
for all connected $d$-regular graphs $G$ of girth at least $g-1$ which are not equal to $H$.
\end{defn}

The only non-trivial property of $K_{d,d}$ that we require for the proof of Theorem~\ref{thmMain}
is that it is strictly $(d,4)$-optimal.

\begin{lemma}\label{lemkddopt}
For $d\ge2$,
$K_{d,d}$ is strictly $(d,4)$-optimal.
\end{lemma}
\begin{proof}
Let $G$ be a connected $d$-regular graph on $n$ vertices and let
$f$ be an injective homomorphism from $C_4$ to $G$.
Write $V(C_4)=\{1,2,3,4\}$.
There are at most $n$ choices for $f(1)$. Given any choice of $f(1)$ there are at most $d(d-1)^2$ choices for $f(\{2,3,4\})$
with equality if and only if each pair of vertices in the neighborhood of $f(1)$ has $d$ common neighbors, i.e., $G=K_{d,d}$.
It follows that if $G$ is not isomorphic to $K_{d,d}$ then 
\[
\frac{1}{n}\inj(C_4,G)\leq d(d-1)^2-1=\frac{1}{2d}\inj(C_4, K_{d,d})-1\, .\qedhere
\]
\end{proof}

More generally, all girth $g$, $d$-regular \emph{Moore graphs} are strictly $(d,g)$-optimal, including $K_{d,d}$ and the clique $K_{d+1}$ for $g=4,3$ respectively (see Section~\ref{secConclude}). 

The following theorem states that for even girth $g$, a disjoint union of copies of a strictly $(d,g)$-optimal graph will maximize $i_k$ and $m_k$ over $d$-regular graphs of girth at least $g-1$ for small enough $k$. 

\begin{theorem}
\label{thmGirth}
Let $g\geq 4$ be even, let $d\geq 2$, and suppose a strictly $(d,g)$-optimal graph $H$ exists. 
Then there exists $\eps=\eps(d,g)>0$ so that the following holds. 
Let $h = |V(H)|$ and suppose $G$ is a $d$-regular graph of girth at least $g-1$ on $n$ vertices where $ n$ is divisible by $h$. 
Then for $k\leq \eps n$,
\[
i_k(G)\leq i_k (H_n) \text { and } m_k(G)\leq m_k (H_n)  \, ,\]
where $H_n$ consists of $n/h$ copies of $H$. 
Moreover, if $G$ is not isomorphic to $H_n$ then the inequality is strict for $k\ge g$.
\end{theorem}

\noindent
Theorem~\ref{thmMain} follows from Theorem~\ref{thmGirth}, Lemma~\ref{lemkddopt}, 
and the following result of the current authors and Roberts which shows 
that Conjectures~\ref{conjKahn} and~\ref{conjUM} hold when $n$ is large and $k\ge \eps n$.
 \begin{theorem}[\cite{DaviesCoefficients}]
  \label{thmStability}
  For all $\eps> 0$ and $d \ge 2$, there is $N_1 = N_1(\eps, d)$ large enough so that for all $n \ge N_1$, $n$ divisible by $2d$, all $d$-regular graphs $G$ on $n$ vertices, and all $k \ge \eps n$,
  \begin{align*}
    m_k(G) & \le m_k(H_{d,n}) \quad \text{and} \quad
    i_k(G) \le i_k(H_{d,n}) \, .
  \end{align*}
  Moreover the inequalities are strict if $k\leq n/2$ and $G$ is not isomorphic to $H_{d,n}$.
\end{theorem}
We remark that the theorem statement in \cite{DaviesCoefficients} does not include the final observation regarding strict inequalities. However it is clear from the proof of Theorem~\ref{thmStability} that this holds.

\begin{proof}[Proof of Theorem~\ref{thmMain}]
Let $d\geq 2$ and let $\eps=\eps(d,4)$ be as in Theorem~\ref{thmGirth}.
Let $N=N_1(\eps,d)$ be as in Theorem~\ref{thmStability}. 
Let $G$ be a $d$-regular graph on $n$ vertices where
$n$ is divisible by $2d$ and $n\geq N$.
Moreover suppose that $G$ is not isomorphic to $H_{d,n}$.

By Lemma~\ref{lemkddopt} and Theorem~\ref{thmGirth} with $g=4$ (noting that all graphs have girth at least $g-1=3$),
for $k\leq \eps n$ we have 
\begin{align}\label{eqikmk}
i_k(G)\leq i_k (H_{d,n}) \text { and } m_k(G)\leq m_k (H_{d,n})  \, ,
\end{align}
 where the inequality is strict for $k\geq 4$.
 By Theorem~\ref{thmStability}, 
 the inequalities~\eqref{eqikmk} hold and are strict for $\eps n\le k \le n/2$ also.
\end{proof}

It remains to prove Theorem~\ref{thmGirth}.
Henceforth let $g\geq 4$ be even and let $d\geq 2$ be such that a $(d,g)$-optimal graph $H$ exists. 
Let $h$ denote the number of vertices in the graph $H$, suppose that $h$ divides $n$, and let $H_n$ be a disjoint union of $n/h$ copies of $H$.

The key idea in proving Theorem~\ref{thmGirth} is to show that the number of independent sets (or matchings) of small size (that is, at most $\eps n$) is essentially determined by the density of the cycles $C_{g-1}$ and $C_g$ in $G$, with more copies of $C_{g-1}$ leading to fewer independent sets (or matchings), and more copies of $C_g$ leading to a greater number. Intuitively, even cycles are beneficial to the count whereas odd cycles are harmful. 
To make this rigorous we write $\log i_k(G)$ and $\log m_k(G)$ as an infinite series using the cluster expansions for the canonical ensembles of the hard-core model and the monomer-dimer model. The first step is to rewrite $i_k(G)$, $m_k(G)$ in terms of the partition function of a \emph{polymer model}. 
The derivation of this polymer model is simple and we present a self-contained argument in the next section; see also~\cite{pulvirenti2012cluster,sokal2001bounds,borgs2006absence} for related work.  We will present the argument first for independent sets and then extend this argument to matchings in Section~\ref{secMatchings}.

\section{Cluster expansion in the canonical ensemble}
\label{secCluster}

Given a simple $d$-regular graph $G=(V,E)$ on $n$ vertices, we begin by showing how $i_k(G)$ can be expressed in terms of the partition function of an \emph{abstract polymer model}. 
Let $G^\circ=(V,E^\circ)$ be $G$ with a self-loop added to every vertex.
Then
\begin{align*}
  i_k(G)
   & =\frac{1}{k!}\sum_{\phi: [k]\to V} \prod_{ij\in \binom{[k]}{2}}\mathbf 1_{\phi(i)\phi(j)\notin E^\circ}                           \\
   & =\frac{1}{k!}\sum_{\phi: [k]\to V} \prod_{ij\in \binom{[k]}{2}}(1+\mathbf 1_{\phi(i)\phi(j)\notin E^\circ}-1)                     \\
   & =\frac{1}{k!}\sum_{\phi: [k]\to V} \sum_{F\subseteq \binom{[k]}{2}}\prod_{ij\in F}(\mathbf 1_{\phi(i)\phi(j)\notin E^\circ}-1)\,.
\end{align*}

Let $\Pi$ denote the set of all unordered partitions of the set $[k]$.
For a subset $S\subseteq[k]$, write $\cC_S$ for the set of all connected graphs on vertex set $S$.
By grouping graphs $([k],F)$ according to their component structure we have, for any fixed $\phi:[k]\to V$,
\begin{align*}
  \sum_{F\subseteq \binom{[k]}{2}}\prod_{ij\in F}(\mathbf 1_{\phi(i)\phi(j)\notin E^\circ}-1)
   & = \sum_{\pi\in \Pi} \prod_{S\in \pi} \sum_{F\in \cC_S}\prod_{ij\in E(F)}(\mathbf 1_{\phi(i)\phi(j)\notin E^\circ}-1)\, .
\end{align*}
This allows us to break the sum over $\phi:[k]\to V$ into separate terms $\phi:S\to V$ for each $S\in\pi\in \Pi$. 
For $S\subseteq[k]$, let
\begin{align}\label{eqweightdef}
  w_G(S):=\frac{1}{n^{|S|}}\sum_{\phi: S\to V} \sum_{F\in \cC_S}\prod_{ij\in E(F)}(\mathbf 1_{\phi(i)\phi(j)\notin E^\circ}-1)\, .
\end{align}
We will refer to $w_G(S)$ as the \emph{weight} of $S$ (with respect to $G$).
Then
\begin{align*}
  i_k(G)
   & =\frac{n^k}{k!} \sum_{\pi\in \Pi} \prod_{S\in \pi} w_G(S)\, .
\end{align*}

Let $\cP$ denote the set of all subsets $S\subseteq [k]$ with $|S|\geq 2$.
We call the elements of $\mathcal P$ \emph{polymers}.
We say that two polymers $S_1, S_2$ are \emph{compatible}, and write $S_1\sim S_2$,
if $S_1\cap S_2=\emptyset$.
We note that $w_G(S)$ depends only on the size of $S$, however it will be important to make $S$ explicit in the notation since the notion of compatibility of polymers depends on the polymers themselves.
Let $\Omega$ denote collection of all sets of pairwise compatible polymers.
Since $w_G(S)=1$ whenever $|S|=1$ we have
\begin{align}\label{eqikcanon}
  i_k(G)
   & =\frac{n^k}{k!} \Xi_k(G)\, ,
\end{align}
where
\[
\Xi_k(G) := \sum_{\Gamma\in\Omega} \prod_{S\in \Gamma} w_G(S)\,.
\] 
The set $\cP$ together with the compatibility relation `$\sim$' and
weight function $w_G$ defines a \emph{polymer model}, as defined by Koteck\'{y} and Preiss~\cite{kotecky1986cluster}, generalizing a technique used to study statistical mechanics models on lattices (e.g.~\cite{gruber1971general}). 
The expression $\Xi_k(G)$ is known as the \emph{polymer model partition function}.  

Expressing $i_k(G)$ as a polymer model partition function (scaled by $n^k/k!$) allows
us to use the \emph{cluster expansion},  an infinite series representation of $\log \Xi_k(G)$. 
To introduce the cluster expansion we require some notation.

Suppose that $\Gamma=(S_1, \ldots, S_t)$ is an ordered tuple of polymers. 
We define the \emph{incompatibility graph} $I(\Gamma)$ to be 
the graph on vertex set ${1,\ldots, t}$ where
$i\sim j$ if and only if $i \neq j$ and $S_i$ is incompatible with $S_j$.
A \emph{cluster} is an ordered tuple $\Gamma$ of polymers whose incompatibility graph $I(\Gamma)$ is connected. 
Given a graph $F$, we define the  \emph{Ursell function} $\phi(F)$ of $F$ to be
\begin{align*}
\phi(F) &= \frac{1}{|V(F)|!} \sum_{\substack{A \subseteq E(F)\\ \text{spanning, connected}}}  (-1)^{|A|} \, .
\end{align*}

Let $\cC$ be the set of all clusters. The \emph{cluster expansion} is the formal power series in the weights $w_G(S)$
\begin{align*}
\log \Xi_k(G) &= \sum_{\Gamma \in \cC} w_G(\Gamma) \,,
\end{align*}
where
\begin{align*}
w_G(\Gamma) &=  \phi(I(\Gamma)) \prod_{S \in \Gamma} w_G(S) \,.
\end{align*}

A sufficient condition for the convergence of the cluster expansion is given by a theorem of Koteck\'y and Preiss.
\begin{theorem}[\cite{kotecky1986cluster}]
\label{thmKP}
Let $a : \cP \to [0,\infty)$ and $b: \cP \to [0,\infty)$ be two functions.  If for all polymers $S \in \cP$, 
\begin{align}
\label{eqKPcond}
\sum_{S' \nsim S}  |w(S')| e^{a(S') +b(S')}  &\le a(S) \,,
\end{align}
then the cluster expansion converges absolutely.  Moreover, if we let $b(\Gamma) = \sum_{S \in \Gamma} b(S)$ and write $\Gamma \nsim S$ if there exists $S' \in \Gamma$ so that $S \nsim S'$, then for all polymers $S$,
\begin{align}
\label{eqKPtail}
 \sum_{\Gamma \in \cC, \;  \Gamma \nsim S} \left |  w(\Gamma) \right| e^{b(\Gamma)} \le a(S) \,.
\end{align}
\end{theorem}

The main result of this section is that  for small values of $k$, as in $k\leq \eps n$ for some explicit $\eps=\eps(d)$, the cluster expansion of $\log \Xi_k(G)$ converges.  

For a cluster $\Gamma$, we let $|\Gamma|$ denote the number of polymers (with multiplicity) in $\Gamma$ and let $\| \Gamma \|$ denote the sum of the sizes of these polymers.

\begin{lemma}
\label{lemConverge}
  For every $d$-regular graph $G$ on $n$ vertices, and all $k \le e^{-5}  n/(d+1)$, the cluster expansion for $\log \Xi_k (G)$ converges absolutely.  Moreover, we have the following tail bound on the cluster expansion.
\begin{equation}
\sum_{\substack{\Gamma \in \cC \\ \| \Gamma \| - |\Gamma| \ge t}}  | w_G(\Gamma)| \le k \gamma^{t}
\end{equation}
where $\gamma =(d+1)e^5 k/n\le 1$.
\end{lemma}

We begin with a standard argument to upper bound the absolute value of the weight $w_G(S)$ of a polymer $S$.
Throughout this section we denote the weight function $w_G$ simply by $w$.
Recall that for a set $S$, we let $\cC_S$ denote the set of all connected graphs with vertex set $S$.
Let us also define $\cT_S$ to be the set of all spanning trees on vertex set $S$.
We appeal to the following special case of an inequality due to Penrose \cite{penrose1967convergence}
(see also \cite[Proposition 4.1]{sokal2001bounds}).  

\begin{lemma}[Penrose inequality \cite{penrose1967convergence}]\label{lempen}
Let $S$ be a finite set and for each $e\in \binom{S}{2}$ let $w_e$ be a complex number such that $|1+w_e|\leq 1$. Then 
\[
\left |\sum_{F\in \cC_S}\prod_{e\in E(F)}w_e \right|\leq \sum_{F\in \cT_S}\prod_{e\in E(F)}|w_e|\, . 
\]
\end{lemma}

For fixed $S\subseteq [k]$ and $\phi: S\to V$, applying Lemma~\ref{lempen} with $w_{ij}=1_{\phi(i)\phi(j)\notin E^\circ}-1$
yields
\[
  \left| \sum_{F\in \cC_S}\prod_{ij\in E(F)}(\mathbf 1_{\phi(i)\phi(j)\notin E^\circ}-1)\right|
  \leq \sum_{F\in \cT_S}\prod_{ij\in E(F)}\mathbf 1_{\phi(i)\phi(j)\in E^\circ}\, .
\]
It follows that 
\begin{align*}
  |w(S)| & \leq \frac{1}{n^{|S|}}\sum_{\phi: S\to V} \left|\sum_{F\in \cC_S}\prod_{ij\in E(F)}(\mathbf 1_{\phi(i)\phi(j)\notin E^\circ}-1)\right| \\
        & \leq \frac{1}{n^{|S|}}\sum_{\phi: S\to V} \sum_{F\in \cT_S} \prod_{ij\in E(F)}\mathbf 1_{\phi(i)\phi(j)\in E^\circ}                 \\
        & \leq \sum_{F\in \cT_S}t(F,G^\circ)\,,
\end{align*}
where we write
\[ t(F,G^\circ) := n^{-|V(F)|}\sum_{\phi:V(F)\to V(G^\circ)}\prod_{ij\in E(F)}\mathbf 1_{\phi(i)\phi(j)\in E(G^\circ)} \]
for the standard notion of \emph{homomorphism density} of $F$ in $G^\circ$.

\begin{lemma}\label{lem:homcounts2}
  Let $G$ be a simple $d$-regular graph on $n$ vertices, and $G^\circ$ be formed from $G$ by adding a self-loop at every vertex. Then for any tree $F$,
  \[ t(F, G^\circ)  = \left(\frac{d+1}{n}\right)^{|V(F)|-1}\,. \]
\end{lemma}
\begin{proof}
The proof is a simple induction on $|V(F)|$. The case $|V(F)|=1$ holds trivially. 
Suppose that $F$ is a tree on at least $2$ vertices and let $x$ be a leaf of $F$.
By the induction hypothesis 
\[ t(F-x, G^\circ)= \left(\frac{d+1}{n}\right)^{|V(F)|-2}\,. \]
The result follows by noting that any homomorphism from $F-x$ to $G^\circ$ can be 
extended to a homomorphism from $F$ to $G$ in precisely $d+1$ ways. 
\end{proof}

Now we can prove Lemma~\ref{lemConverge}.

\begin{proof}[Proof of Lemma~\ref{lemConverge}]

Applying Lemma~\ref{lem:homcounts2} together with Cayley's formula, the well-known fact that $|\cT_S|=|S|^{|S|-2}$, we obtain
\begin{equation}\label{eq:weightbound}
  |w(S)|\leq |S|^{|S|-2}\left(\frac{d+1}{n}\right)^{|S|-1}\,.
\end{equation}

Given this bound, we can now verify the Koteck\'y--Preiss condition~\eqref{eqKPcond} 
with $a(S) = |S|$ and $b(S) = K ( |S| -1)$
where $K = \log \frac{ n   }{ (d+1) e^5k  }  \ge0$ .
We want to show that for all $S'\in \cP$
\[
  \sum_{S\not\sim S'}|w(S)|e^{(K+1)|S| -K} \leq |S'|\, ,
\]
and hence it suffices to show that for all $v\in [k]$,
\[
  \sum_{S\ni v}|w(S)|e^{(K+1)|S| -K} \leq 1\, .
\]
Now the weight bound~\eqref{eq:weightbound} gives
\begin{align*}
  \sum_{S\ni v}|w(S)|e^{(K+1)|S| -K}
   & \leq \sum_{j=2}^k \binom{k}{j-1}j^{j-2}\left(\frac{d+1}{n}\right)^{j-1}e^{(K+1)j-K}                    \\
   & \leq \sum_{j=2}^k \left(\frac{ek}{j-1}\right)^{j-1}j^{j-2}\left(\frac{d+1}{n}\right)^{j-1}e^{(K+1)j-K} \\
   & \leq e^{2} \sum_{j=2}^\infty \left(\frac{e^{(K+2)}k(d+1)}{n}\right)^{j-1}                              \\
   &= e^2\sum_{j=1}^\infty  e^{-3j} < 1 \,.
\end{align*}
Theorem~\ref{thmKP} then tells us that the cluster expansion converges absolutely, and applying~\eqref{eqKPtail} to the polymer $S=[k]$ which is incompatible with every $S'\in\cP$ gives
\begin{align*}
\sum_{\Gamma \in \cC}  |w_G(\Gamma)|  e^{K(\|\Gamma\| - |\Gamma|)} \le k
\end{align*}
which implies 
\begin{equation*}
\sum_{\substack{\Gamma \in \cC \\ \| \Gamma \| - |\Gamma| \ge t}}  | w_G(\Gamma)| \le k \gamma^{t}
\end{equation*}
where $\gamma =e^{-K} = (d+1) e^5 k/n$.
\end{proof}

\section{Proof of Theorem~\ref{thmGirth} for independent sets}
\label{secGirth}

We will prove Theorem~\ref{thmGirth} in several steps.    Throughout this section we assume that $H$ is strictly $(d,g)$-optimal on $h$ vertices,  $G$ is a $d$-regular graph on $n$ vertices of girth at least $g-1$, and $H_n$ is the union of $n/h$ copies of $H$.  It will be useful to introduce the following notation.
For a graph $F$, let
\[
t(F):= t(F, H_n^\circ)- t(F, G^\circ)\, .
\]
Lemma~\ref{lem:homcounts2} tells us that $t(F)= 0$ for any tree $F$, and since $G$ and $H_n$ have girth at least $g-1$ we have $t(F) =0$ for any $F$ on at most $g-2$ vertices.  
We also have $t(C_{g-1})\le 0$ which follows from the fact that $H_n$ has girth at least $g$ whereas $G$ is permitted to have girth $g-1$.

First we will deal with the relatively simple case of $k \le  g-1$. 
In this case,
\begin{align*}
i_k(H_n)-i_k(G)&=\frac{n^k}{k!}\left(\Xi_k(H_n)- \Xi_k(G) \right)\\
&= \frac{n^k}{k!}\sum_{\Gamma \in \Omega} \prod_{S\in \Gamma}\sum_{F\in \cC_S}(-1)^{|E(F)|}t(F)\\
&= - \frac{n^k}{k!} t(C_{g-1})\mathbf 1_{k=g-1}\geq 0.
\end{align*}
We henceforth assume that $k\geq g$.  We will write $G = G_0 \cup G_H$ where $G_H$ is a union of copies of $H$ and $G_0$ has no component isomorphic to $H$.  We let $\alpha = |G_0|/n$ and can also write $H_n = H_0 \cup G_H$ where $H_0$ is the graph on $\alpha n$ vertices consisting of a union of copies of $H$. 

We first consider the case $\alpha \ge 1/10$; in this case there is a significant gap in the $C_g$ density of $G$ and $H_n$ which will be enough to prove the result via the canonical ensemble cluster expansion of the previous section.

\begin{lemma}
\label{lemGirthGap}
There exists $\delta=\delta(d,g)$ so that  if $\alpha \ge 1/10$,
\[
t(C_g)\geq \delta n^{1-g} + \frac{g}{n}t(C_{g-1}) \,.
\]
\end{lemma}
\begin{proof}
For a graph $F$ and unordered partition $\pi$ of $V(F)$, we denote by $F/\pi$,
the graph obtained by identifying nodes that belong to the same part of $\pi$ and then deleting loops and multiple edges.
We have the following relation:
\begin{align}\label{eqcontract}
\hom(F, G^\circ)= \sum_{\pi\in\Pi}\inj(F/\pi, G)\, ,
\end{align}
 where the sum ranges over the unordered partitions $\Pi$ of $V(F)$.
 Letting $F=C_g$, and noting that $G$ has girth at least $g-1$,
 we have that $\inj(F/\pi, G)\neq 0$ only if 
 $F/\pi$ is a tree or a cycle of length $g$ or $g-1$.
 If $F/\pi=T$ is a tree then it must have at most $g/2$ edges
 (indeed if $U$, $W$ are sets in the partition $\pi$ corresponding to an edge of $T$, then $U,W$ must be connected by at least two edges of $F$, else $F=C_g$ could be disconnected by the removal of one edge).
 If $\{x,y\}$ is an edge of $T$ where $x$ is a leaf, then
 since $G$ has girth at least $g-1$ we have
 \[
 \inj(T, G)=\inj(T-x, G)(d-d_{T-x}(y))\, ,
 \]
where $d_{T-x}(y)$ denotes the degree of $y$ in the graph $T-x$.
Iterating the above, it follows that $\inj(T, G)$ depends only on $d$, $n$ and $T$.
We therefore have
\[
\hom(C_g, G^\circ)= \inj(C_g, G)+ g\inj(C_{g-1}, G) + f(n,d,g)\, ,
\]
where $f$ is a function only of $n,d$ and $g$. The second term appears with coefficient $g$ as contracting an edge of $C_g$ is the only way for $F/\pi$ to equal $C_{g-1}$. 
Similarly, noting that $\inj(C_{g-1}, H_n)=0$, we have
\[
\hom(C_g, H_n^\circ)= \inj(C_g, H_n) + f(n,d,g)\, ,
\]
for the same function $f$ and so 
 \begin{align}\label{eqtinj}
 t(C_g)=n^{-g}\bigl[ \inj(C_g, H_n)- \inj(C_g, G)-g\inj(C_{g-1}, G) \bigr]\, .
 \end{align}
  
  Let $G_1,\dotsc,G_r$ be the connected components of $G$. 
  We have that $\inj(C_g,G)=\sum_{j=1}^r\inj(C_g,G_j)$ because $C_g$ is connected. 
  When $G_j$ is not isomorphic to $H$ we have 
  \[ \inj(C_g,G_j) \le v(G_j)\left(\frac{1}{h}\inj(C_g,H) -\eta\right) \]
  by the strict $(d,g)$-optimality of $H$. 
 This gives 
  \begin{align}\label{eqCggap}
   \inj(C_g,G) \le \frac{n}{h}\inj(C_g,H) - \alpha n \eta \le \inj(C_g,H_n) -\frac{n}{10}\eta \,.
 \end{align}
Letting $\delta=\eta/10$, it follows from~\eqref{eqtinj} that
\[
t(C_g)\geq \delta n^{1-g}-g\inj(C_{g-1}, G)n^{-g} \,.
\]
The argument used to derive \eqref{eqtinj} shows also that 
\[
t(C_{g-1})=-n^{1-g}\inj(C_{g-1}, G)\, .
\]
The result follows. 
\end{proof}

We now deduce the case $\alpha \ge 1/10$ of Theorem~\ref{thmGirth} in a strong form.
\begin{lemma}
\label{lemGirthNoH}
There exist $\eps_1=\eps_1(d,g)>0$ and $c=c(d,g)>0$ so that if $\alpha \ge 1/10$ and  $g\le k\le \eps_1 n$, then
\[
i_k(G)\leq \exp\{-c k^g n^{1-g}\} i_k (H_n)\, .
\]
\end{lemma}

\begin{proof}
Recall that $|\Gamma|$ denotes the number of polymers in a cluster $\Gamma$ and $\| \Gamma\|$ denotes the sum of their sizes.  For $k \le e^{-5} n/(d+1)$, we can apply Lemma~\ref{lemConverge} with $t= g$ to obtain
\[
\sum_{\Gamma: \|\Gamma\| - |\Gamma| \leq g-1}w_G(\Gamma) - k \gamma^g \le \log \Xi_k(G) \le \sum_{\Gamma: \|\Gamma\| - |\Gamma| \leq g-1}w_G(\Gamma) + k \gamma^g
\]
with $\gamma = (d+1) e^5 k/n$. 
If $\Gamma$ is a cluster such that $\|\Gamma\| - |\Gamma| \leq g-1$, then one of the following holds:
\begin{enumerate}
\item Each polymer in $\Gamma$ has size $\leq g-2$,
\item $\Gamma=(S)$ where $|S|\in \{g-1, g\}$,
\item $\Gamma=(S_1, S_2)$ where $\{|S_1|, |S_2|\}=\{g-1, 2\}$.
\end{enumerate}

We note that if each polymer in $\Gamma$ has size at most $g-2$,
then by Lemma~\ref{lem:homcounts2}, 
$w_G(\Gamma)$ is the same for all $d$-regular graphs $G$ on $n$ vertices with 
girth at least $g-1$.
It follows that
\begin{equation}
\label{eqLogRatio}
\begin{aligned}
\log \frac{i_k(H_n)}{i_k(G)} &= \log\frac{\Xi_k(H_n)}{\Xi_k(G)} \\&\ge \sum_{\substack{\Gamma = (S)\\ |S|\in\{g-1,g\}}} (w_{H_n}(\Gamma)-w_G(\Gamma))+ \sum_{\substack{\Gamma = (S_1, S_2)   \\\mathclap{\{|S_1|, |S_2|\}=\{g-1, 2\}}}}  (w_{H_n}(\Gamma)-w_G(\Gamma))  - 2k \gamma^g\, .
\end{aligned}
\end{equation}

If $S$ is a polymer of size $g-1$, then by Lemma~\ref{lem:homcounts2} $t(F)=0$ for all graphs $F\in \cC_S$ except those isomorphic to $C_{g-1}$, giving for $\Gamma=(S)$ that
\[
w_{H_n}(\Gamma)-w_G(\Gamma)= (-1)^{g-1}\frac{(g-2)!}{2}t(C_{g-1})=-\frac{(g-2)!}{2}t(C_{g-1})\, .
\]

Suppose $S$ is a polymer of size $g$,
and suppose also that $g\geq 6$. 
Then any graph $F\in \cC_S$ of girth at least $g-1$ is either a tree, 
a cycle of length $g-1$ with a pendant edge or a cycle of length $g$.
If $g=4$, then we have two additional possibilities for a graph
$F\in \cC_S$: $F=\Dgraph$ and $F=\Kfourgraph$.

If $F$ is a cycle of length $g-1$ plus a pendant edge then by~\eqref{eqcontract}
we have
\[
|t(F)|=n^{-g}[\inj(F,G)+(g-1)\inj(C_{g-1},G)] \leq n^{-g}(d+g)\inj(C_{g-1}, G)\leq \frac{d+g}{n}|t(C_{g-1})|\, .
\]
Similarly, if $F=\Dgraph$ or $F=\Kfourgraph$ (and $g=4$) then
\[
t(F)= O(1/n)t(C_{3})
\]
where the implied constant depends only on $d$. 
It follows that if $\Gamma=(S)$, then 
\[
w_{H_n}(\Gamma)-w_G(\Gamma)= \frac{(g-1)!}{2}t(C_g)+ O\left(\frac{1}{n}\right)t(C_{g-1})\, ,
\]
where the implied constant depends only on $d$ and $g$.
Finally, if $\Gamma=(S_1, S_2)$ where $\{|S_1|, |S_2|\}=\{g-1, 2\}$,
then 
\[
w_{H_n}(\Gamma)-w_G(\Gamma)=-\frac{d+1}{4n}(g-2)!t(C_{g-1})\, .
\]

From~\eqref{eqLogRatio} we now have
\begin{align*}
\log\frac{i_k(H_n)}{i_k(G)}&\ge
\begin{aligned}[t]\frac{(g-2)!}{2}\Bigg[&-\binom{k}{g-1}t(C_{g-1})+
(g-1)\binom{k}{g}\left( t(C_g)+ O\left(\frac{1}{n}\right)t(C_{g-1})\right)\\
&-\binom{k}{g-1}(g-1)(k-1)\frac{d+1}{n} t(C_{g-1})\Bigg] - 2k \gamma^g\, .
\end{aligned}\\
&\geq  \frac{(g-1)!}{2}\binom{k}{g}t(C_g)-\frac{(g-2)!}{3}\binom{k}{g-1}t(C_{g-1}) - 2k \gamma^g\, ,
\end{align*}
provided $\eps:=k/n$ is sufficiently small (recall that $t(C_{g-1})\le 0$), which small enough $\eps_1=\eps_1(d,g)$ and $k\le \eps_1 n$ permits. 

Since $\alpha \ge  1/10$, Lemma~\ref{lemGirthGap} gives
\begin{align}
  \log\frac{i_k(H_n)}{i_k(G)}&\geq 
\frac{(g-1)!}{2}\binom{k}{g} 
\left(\delta n^{1-g} +\frac{g}{n}t(C_{g-1})\right)
- \frac{(g-2)!}{3}\binom{k}{g-1}t(C_{g-1})-2 k  \left  ( \frac{  (d+1)e^5 k } {n  }  \right)^g \nonumber\\
&= \Omega( \eps^g n )   + O \left ( \eps^g n^{g-1}\right) t(C_{g-1}) - \Omega\left(\eps^{g-1} n^{g-1}      \right ) t(C_{g-1} ) - O \left( \eps^{g+1}  n  \right) \label{eqloggap}
\end{align}
where the $\Omega(\cdot)$ and $O(\cdot)$ terms are all positive and the implied constants depend only on $d$ and $g$. 
The result follows provided $\eps_1$ is sufficiently small. 
\end{proof}

We now complete the proof of Theorem~\ref{thmGirth} by treating the case  $\alpha < 1/10$.  
The intuition behind the proof is that a uniformly chosen independent set of size $k$ in $G$
should intersect $G_0$ in approximately $\tfrac{k}{n}|G_0|$ vertices with high probability.
In this case we can apply Lemma~\ref{lemGirthNoH} to get the desired inequality.

\begin{proof}[Proof of Theorem~\ref{thmGirth} for independent sets]
By the above results we may assume $k \ge g$ and $0 < \alpha < 1/10$. 
We will show that there exists $\eps_2=\eps_2(d,g)$ sufficiently small such that if $g \le k\le \eps_2 n$ then $i_k(G)< i_k(H_n)$.

With $a \wedge b$ denoting the minimum of $a$ and $b$, we can write
\begin{align*}
 i_k(G) &= \sum_{j=0}^{ k \wedge \alpha n} i_j(G_0) i_{k-j}(G_H) \\
 i_k(H_n) &= \sum_{j=0}^{ k \wedge \alpha n} i_j(H_0) i_{k-j}(G_H) \,.
\end{align*}
Let $\eps_1 = \eps_1(d,g)$ be the constant from Lemma~\ref{lemGirthNoH}.  
Since $i_j(H_0) \ge i_j(G_0)$, for all $j \le \lfloor \eps_1\alpha n\rfloor$,
\begin{align} \label{eqikdiff}
i_k(H_n)-i_k(G)
\geq \sum_{j=\lfloor \eps_1\alpha n\rfloor}^{k \wedge \alpha n}  (i_j(H_0)-i_j(G_0))i_{k-j}(G_H)\, .
\end{align}

We consider two cases.  
Suppose first that $\lfloor \eps_1\alpha n\rfloor \leq g$, and note that this implies via~\eqref{eqikdiff} and the simple cases $k< g$ applied to $H_0$ and $G_0$ that in fact
\begin{align*}
  i_k(H_{n}) - i_k(G) & \ge \sum_{j=g}^{k \wedge \alpha n} (i_j (H_0) - i_j(G_0)) i_{k-j}(G_H)  \,.
\end{align*}
By Lemma~\ref{lemGirthNoH} we have $i_{g} (H_0) > i_{g}(G_0)$, and these are both natural numbers so differ by at least one. 
Hence it suffices to show that
\begin{align}\label{eqsumg}
i_{k-g}(G_H) > \sum_{j=g+1}^{k \wedge \alpha n} (i_j (G_0) - i_j(H_0)) i_{k-j}(G_H) \,.
\end{align}

We will use the following fact, shown by counting the ways of extending each independent sets of size $t$ in a $d$-regular graph $G$ on $n$ vertices. For each such independent set there are at most $n-t$ and at least $n-(d+1)t$ ways of extending it by a single vertex. Then whenever $n > (d+1)t$ we have
\[   \frac{t+1}{n} \le \frac{t+1}{n-t} \le \frac{i_t(G)}{i_{t+1}(G)} \le \frac{t+1}{n - (d+1)t}  \,.  \]

When $\eps_2$ is small enough in terms of $d$, for $t\leq \eps_2 n$ we therefore have
\begin{align}\label{eqfreevol}
  \frac{i_{t-1}(G_H)}{i_t(G_H)}\leq \frac{t}{(1-\alpha)n-(d+1)(t-1)}\leq 2\eps_2\, ,
\end{align}
(using that $\alpha\le 1/10$).
It follows that the right hand side of \eqref{eqsumg} is at most
\begin{align*}
  i_{k-g}(G_H)\sum_{j=g+1}^{k \wedge \alpha n}  \binom{\alpha n}{j}(2\eps_2)^{j-g}< i_{k-g}(G_H)\sum_{j=g+1}^{k \wedge \alpha n}  \binom{g/\eps_1}{j}(2\eps_2)^{j-g}< i_{k-g}(G_H)\, ,
\end{align*}
for $\eps_2$ sufficiently small and so
\eqref{eqsumg} holds.

Suppose now that $\lfloor \eps_1\alpha n\rfloor > g$. 
When $t< k \wedge \alpha n$ and $k\le \eps_2 n$ for $\eps_2$ small enough in terms of $d$, we have
\begin{align*}
\frac{i_{t+1}(G_0)i_{k-t-1}(G_H)}{i_t(G_0)i_{k-t}(G_H)}
&\le \frac{\alpha n -t}{t+1}\frac{k-t}{(1-\alpha)n-(d+1)(k-t-1)}\\
&\le \frac{\alpha n}{t+1}\cdot \frac{k}{(1-\alpha)n-(d+1)k}\,,
\end{align*}
which is decreasing as $t$ increases, so when we also have $t\ge \alpha \eps_1 n$ (and we use that $\eps_2$ is small enough in terms of $d$ and $k\le\eps_2 n$), this gives
\[ \frac{i_{t+1}(G_0)i_{k-t-1}(G_H)}{i_t(G_0)i_{k-t}(G_H)} \le 2\frac{\eps_2}{\eps_1}\,.\]
Thus for $\ell=\lfloor \eps_1\alpha n\rfloor$ and $j\ge 0$ such that $\ell+j \le k\wedge \alpha n$ we have
\[ i_{\ell+j}(G_0)i_{k-\ell-j}(G_H) \le (2\eps_2/\eps_1)^j \cdot i_{\ell}(G_0)i_{k-\ell}(G_H)\,.\] 

Now working from~\eqref{eqikdiff} we have
\begin{align*}
  i_k(H_n)-i_k(G) 
     &\ge (i_\ell(H_0)-i_\ell(G_0))i_{k-\ell}(G_H) - \sum_{j=1}^{(k\wedge\alpha n) - \ell}i_{\ell+j}(G_0)i_{k-\ell-j}(G_H)
  \\ &\ge (i_\ell(H_0)-i_\ell(G_0))i_{k-\ell}(G_H) - \frac{2\eps_2/\eps_1}{1-2\eps_2/\eps_1}i_{\ell}(G_0)i_{k-\ell}(G_H)
  \\ &\ge (i_\ell(H_0)-i_\ell(G_0))i_{k-\ell}(G_H) - 3\frac{\eps_2}{\eps_1}i_{\ell}(G_0)i_{k-\ell}(G_H)\,,
\end{align*}
when $\eps_2$ is small enough in terms of $\eps_1$. 
Since $i_\ell(G_0)i_{k-\ell}(G_H)>0$ it therefore suffices to show that $i_\ell(H_0)/i_\ell(G_0)>1+3\eps_2/\eps_1$.
By Lemma~\ref{lemGirthNoH} with $G_0$ in place of $G$, recalling that $|H_0|=|G_0|=\alpha n$, we have 
\begin{align*}
i_\ell(H_0)/i_\ell(G_0)
&\ge \exp\left\{c\left(\frac{\ell}{\alpha n}\right)^g \alpha n\right\}\\
&> 1+3\eps_2/\eps_1\, ,
\end{align*}
for $\eps_2$ sufficiently small.
\end{proof}

\section{Proof of Theorem~\ref{thmGirth} for matchings}
\label{secMatchings}
The proof of Theorem~\ref{thmGirth} for matchings is essentially the same as that for independent sets subject to 
some straightforward modifications which we detail now. 

Given a graph $G$, the \emph{line graph} of $G$, denoted by $L(G)$,
is the graph on vertex set $E(G)$ where $e\sim f$ if and only if $e,f$ share an endpoint in $G$ (and $e\neq f$).
We then have for all $k$,
\[
m_k(G)= i_k(L(G))\, .
\]
If $G$ is a $d$-regular graph on $n$ vertices then $L(G)$ is a
$D$-regular graph on $N$ vertices where $D=2(d-1)$ and $N= nd/2$. 
We may therefore simply replace $G$ and $H_n$ in the above arguments with their line graphs
and replace $d$ with $D$ and $n$ with $N$. 

The only step that does not follow trivially under this substitution is in establishing Lemma~\ref{lemGirthGap}.
Indeed this is the only step where we apply the $(d,g)$-optimality of $H$.  
We bridge this gap with the following lemma.
Given a graph $F$, let
\[
\tau(F):=t(F, L(H_n)^\circ)-t(F, L(G)^\circ)\, .
\]
We let $\sub(F,G)$ denote the number of copies of $F$ in $G$ as a subgraph.
\begin{lemma}
If at least $10\%$ of the vertices of $G$ belong to components not isomorphic to $H$ then
there exists $\delta=\delta(d,g)$ such that 
\[
\tau(C_{g})\geq \delta N^{1-g}+\frac{gD}{N} \tau(C_{g-1})\, .
\]
Moreover
\[
\tau(C_{g-1})\leq 0 \, .
\]
\end{lemma}
\begin{proof}
In the following $f_1,f_2,\ldots$ represent functions of $n,d$ and $g$ only.
Since $G$ has girth $g-1$, the edge-minimal subgraphs of $G$ whose line graphs contain $C_{g-1}$
are $K_{1,g-1}$ and $C_{g-1}$ itself. 
The line graph of $C_{g-1}$ contains one copy of $C_{g-1}$ whereas the line graph of $K_{1,g-1}$
(a clique on $g-1$ vertices) contains $(g-2)!/2$ copies. Thus
\[
\sub(C_{g-1}, L(G))=\sub(C_{g-1}, G)+\frac{(g-2)!}{2}n\binom{d}{g-1}=\sub(C_{g-1}, G)+f_1\, .
\]
Since $\inj(C_{g-1}, G)=2(g-1)\sub(C_{g-1}, G)$ (and similarly for $L(G)$) we have
\begin{align}\label{eqLCg-1}
\inj(C_{g-1}, L(G))=\inj(C_{g-1}, G)+ f_2\, .
\end{align}
If $F$ is a graph on less than $g-1$ vertices then, since $G$ has girth at least $g-1$, the edge minimal subgraphs
of $G$ that contain $F$ must all be trees. 
It follows that $\inj(F, L(G))$ depends only on $n,d$ and $F$.
By the relation~\eqref{eqcontract} (with $L(G)$ in place of $G$) and~\eqref{eqLCg-1}
we therefore have
\[
\hom(C_{g-1}, L(G)^\circ)=\inj(C_{g-1}, L(G))+f_3=\inj(C_{g-1}, G)+ f_4\, .
\]
Similarly
\[
\hom(C_{g-1}, L(H_n)^\circ)=\inj(C_{g-1}, H_n)+ f_4=f_4
\]
and so
\begin{align}\label{eqtauC}
\tau(C_{g-1})=-N^{1-g}\inj(C_{g-1}, G)\leq 0\, .
\end{align}
This proves the second assertion of the lemma. 

By \eqref{eqcontract} again and~\eqref{eqLCg-1} we have
\begin{align*}
\hom(C_g, L(G)^\circ)&=\inj(C_g, L(G))+ g\inj(C_{g-1}, L(G))+f_5 \\
& =\inj(C_g, L(G))+ g\inj(C_{g-1}, G)+f_6\, .
\end{align*}
Similarly 
\[
\hom(C_g, L(H_n)^\circ)=\inj(C_g, L(H_n))+f_6\, .
\]

The edge-minimal subgraphs of $G$ whose line graphs contain $C_{g}$
 are $K_{1,g}$, $C_{g}$ and $C_{g-1}$ with a pendant edge. 
Let $P$ denote $C_{g-1}$ with a pendant edge. 
The line graphs of $C_g$ and $P$ both contain one copy of $C_g$ and so
 \[
\sub(C_g, L(G))= \sub(C_g, G)+ \sub(P, G) + f_7\,.
\]
We also have 
\[
\sub(P, G)\le (g-1)(D-2)\cdot\sub(C_{g-1},G) \, ,
\]
 by bounding the number of ways of extending a copy of $C_{g-1}$ in $G$ to a copy of $P$.
Therefore
\[
\sub(C_g, L(G))\leq \sub(C_g, G)+(g-1)(D-2)\cdot \sub(C_{g-1}, G) + f_7\, .
\]
Again since $\inj(C_{g}, G)=2g\sub(C_{g-1}, G)$ (and similarly for $L(G)$) we have
\[
\inj(C_g, L(G))\leq \inj(C_g, G)+g(D-2)\cdot \inj(C_{g-1}, G) + f_8\, .
\]
Similarly, noting that $H_n$ is $C_{g-1}$-free, we have
\[
\inj(C_g, L(H_n))= \inj(C_g, H_n)+ f_8\, .
\]

Putting everything together we get
\begin{align*}
N^g\tau(C_g)&= \inj(C_g, L(H_n))-\inj(C_g, L(G))- g\inj(C_{g-1}, G)\\
&\ge \inj(C_g, H_n)-\inj(C_g, G)- gD\inj(C_{g-1}, G)\\
&\ge \frac{\eta}{10} n- gD\inj(C_{g-1}, G)\\
&=\frac{\eta}{5d} N+gDN^{g-1}\tau(C_{g-1})
\, ,
\end{align*}
where in the third line we used \eqref{eqCggap} and for the final equality we used~\eqref{eqtauC}.
\end{proof}

\section{Extensions}\label{secConclude}

In light of Theorem~\ref{thmGirth},  we would like to find examples of strictly $(d,g)$-optimal graphs. 
One class of examples are \emph{Moore graphs}.
If $g$ is even, a $(d,g)$-Moore graph is a $d$-regular graph of girth $g$ with precisely 
\[
1+(d-1)^{(g-2)/2}+d\sum_{j=0}^{(g-4)/2}(d-1)^j
\]
vertices.
If $g$ is odd a $(d,g)$-Moore graph is a $d$-regular graph of girth $g$ with precisely 
\[
1+ d\sum_{j=0}^{(g-3)/2} (d-1)^j
\]
vertices. 
In particular $K_{d,d}$ is a $(d,4)$-Moore graph and $K_{d+1}$ is a $(d,3)$-Moore graph.  It is not difficult to generalize Lemma~\ref{lemkddopt}
to show that for $g\ge 3$ and $d\ge 2$, any $(d,g)$-Moore graph is indeed 
strictly $(d,g)$-optimal (see also~\cite{azarija2015moore}).  Perarnau and Perkins~\cite{perarnau2018counting} conjectured that Moore graphs of even and odd girth are maximizers and minimizers respectively of $\frac{1}{n} \log i(G)$.   Theorem~\ref{thmGirth} gives some evidence towards this conjecture, along with the following minimization result which follows by mimicking the proof of Theorem~\ref{thmGirth} and using the fact that $t(C_g, \cdot)$ terms appear in the cluster expansion with sign depending on the parity of $g$. 
\begin{theorem}
\label{thmGirthMin}
Let $g\geq 3$ be odd, let $d\geq 2$, and suppose a strictly $(d,g)$-optimal graph $H$ exists. 
Then there exists $\eps=\eps(d,g)>0$ so that the following holds. 
Let $h = |V(H)|$ and suppose $G$ is a $d$-regular graph of girth at least $g$ on $n$ vertices where $ n$ is divisible by $h$. 
Then for $k\leq \eps n$,
\[
i_k(G)\geq i_k (H_n) \quad\text{and}\quad m_k(G)\geq m_k (H_n)  \, ,\]
where $H_n$ consists of $n/h$ copies of $H$. 
Moreover, if $G$ is not isomorphic to $H_n$ then the inequality is strict for $k\ge g$.
\end{theorem}

In particular, since $K_{d+1}$ is $(d,3)$-optimal it minimizes the number of matchings of small size, which proves the second statement in Theorem~\ref{thmCliquemin}.

A further example is the analogous result  to Theorem~\ref{thmMain} for cubic graphs of girth at least $5$.
Let $W$ denote the \emph{Heawood graph}: the $(3,6)$-Moore graph on $14$ vertices.  For $n$ divisible by $14$, let $W_n$ be the union of $n/14$ copies of $W$.   In~\cite{perarnau2018counting} it was proved that $W_n$ maximizes the independence polynomial for all values of $\lam$ over all $3$-regular graphs of girth at least $5$.  Here we show that this holds coefficient-by-coefficient for large enough $n$. 
\begin{theorem}
\label{thmHeawood}
For all sufficiently large $n$ divisible by $14$, and all $3$-regular graphs $G$ of girth at least $5$ on $n$ vertices, $ i_k(G) \le i_k(W_n)$ for all $k$.  The inequality is strict for $k \ge6$. 
\end{theorem}
\begin{proof}
Let $\eps = \eps(3,6)$  be as in Theorem~\ref{thmGirth}. Suppose that $G$ is not isomorphic to $W_n$.  In~\cite[Theorem~7]{DaviesCoefficients}, it is shown that there is $N(\eps)$ large enough so that if $n \ge N$ and $k >\eps n$, then $i_k(G) < i_k(W_n)$.  Since $W$ is a $(3,6)$-Moore graph it is strictly $(3,6)$-optimal and so Theorem~\ref{thmGirth} shows that $i_k(G) \le i_k(W_n)$ for $k\le \eps n$ with strict inequality for $k\ge 6$.
\end{proof}

%We conjecture that the following generalization of Theorems~\ref{thmMain} and~\ref{thmHeawood} holds.
%\begin{conj}
 %The conclusion of Theorem~\ref{thmGirth} in holds throughout the entire range $1\le k \le n/2$.
%\end{conj}

Finally, to prove the first statement of Theorem~\ref{thmCliquemin}, we briefly give an argument using the cluster expansion for the grand canonical ensemble partition function $Z^m_G(\lam)$.

\begin{proof}[Proof of the first statement in Theorem~\ref{thmCliquemin}]
Without loss of generality we may assume that $G$ has no $K_{d+1}$ component, as $\log Z^m_G(\lam)$ is additive over components.

Let $L(G)$ be the line graph of $G$, which is a $D$-regular graph for $D=2(d-1)$.
The observation that $Z^m_G(\lam) = Z_{L(G)}(\lam)$ is equivalent to the statement that $Z^m_G(\lam)$ is the partition function of a polymer model where the set of polymers $\cP$ is $E(G)$, $w(e)=\lam$ for $e\in E(G)$, and two edges are incompatible if and only if they are incident (or identical).

A standard application of Theorem~\ref{thmKP} shows that $|\lam| \le 1/(e(D+1))$ is sufficient for convergence of the cluster expansion, and with a little extra room, say $|\lam| \le 1/(Ke(D+1))$ for some $K\ge 1$, we have the tail bound 
\begin{equation}\label{eqMonomerTailBound}
  \sum_{\substack{\Gamma\in\cC\\ |\Gamma| \ge t}}|w(\Gamma)| \le nK^{-t}\,,
\end{equation}
obtained by summing~\eqref{eqKPtail} over a collection of edges that covers every vertex. 

To calculate the first few terms of the cluster expansion, note that clusters of size 1 are single edges, clusters of size 2 are single edges with multiplicity 2 or pairs of edges sharing a vertex, and clusters of size 3 consist of single edges with multiplicity 3, pairs of edges sharing a vertex one of which has multiplicity 2, claws, triangles, and paths of 3 edges.
The weight of any given cluster $\Gamma$ is simply $\lam^{|\Gamma |}$ which does not depend on the graph, but the number of each type of cluster may vary amongst $d$-regular graph.
In fact, for all of these clusters except triangles and paths of 3 edges, the number is constant over $d$-regular graphs.
Recalling that clusters are ordered tuples of polymers (whose incompatibility graph is connected), we have that the number of triangle clusters in $G$ is $\inj(C_3,G)$, and the number of 3-edge path clusters is $3d(d-1)^2n - 3\inj(C_3,G)$.
For triangle clusters $H(\Gamma)$ is a triangle and $\phi(H(\Gamma))=1/3$. 
For 3-edge path clusters $H(\Gamma)$ is a 2-edge path and $\phi(H(\Gamma))=1/6$. 
Then using the tail bound~\eqref{eqMonomerTailBound} we have
\[ \frac{1}{n} \log Z^m_G(\lam) - \frac{1}{d+1} \log Z^m_{K_{d+1}}(\lam) \ge \frac{1}{6}\lam^3 \left [ \frac{\inj(C_3,K_{d+1})}{d+1} -  \frac{\inj(C_3, G)}{n} \right]  - 2K^{-4}\,. \]

The clique $K_{d+1}$ has $\inj(C_3,K_{d+1})/(d+1) = d(d-1)$, while any $n$-vertex $d$-regular graph $G$ without a $K_{d+1}$ component has $\inj(C_3,G)/n \le d(d-1) -2$ as every vertex of $G$ must be contained in at least one triple of vertices that does not induce a triangle. 
Then 
\[ \frac{1}{n} \log Z^m_G(\lam) - \frac{1}{d+1} \log Z^m_{K_{d+1}}(\lam) \ge \frac{1}{3}\lam^3 - 2K^{-4}\,, \]
which is strictly positive for any $\lam$ satisfying 
\[ 0 < \lam < \frac{1}{6(2d-1)^4e^4} \]
by choosing  $K= ( e\lam(2d-1))^{-1}$. That gives the required statement with  $c=1/(96e^4)$.
\end{proof}

\section*{Acknowledgements}
We thank P{\'e}ter Csikv{\'a}ri for sharing a draft of~\cite{csikvari20} with us. We also thank Julian Sahasrabudhe for carefully reading the first version of this paper. WP was supported in part by NSF grants DMS-1847451 and CCF-1934915.

\providecommand{\bysame}{\leavevmode\hbox to3em{\hrulefill}\thinspace}
\providecommand{\MR}{\relax\ifhmode\unskip\space\fi MR }
\providecommand{\MRhref}[2]{%
  \href{http://www.ams.org/mathscinet-getitem?mr=#1}{#2}
}
\providecommand{\href}[2]{#2}


\begin{thebibliography}{10}

\bibitem{alon1991independent}
Noga Alon, \emph{Independent sets in regular graphs and sum-free subsets of
  finite groups}, Israel Journal of Mathematics \textbf{73} (1991), 247--256.

\bibitem{azarija2015moore}
Jernej Azarija and Sandi Klav{\v{z}}ar, \emph{Moore graphs and cycles are
  extremal graphs for convex cycles}, Journal of Graph Theory \textbf{80}
  (2015), 34--42.

\bibitem{balogh2020counting}
J{\'o}zsef Balogh, B{\'e}la Bollob{\'a}s, and Bhargav Narayanan, \emph{Counting
  independent sets in regular hypergraphs}, Journal of Combinatorial Theory,
  Series A \textbf{180} (2021), 105405.

\bibitem{balogh2020independent}
J{\'o}zsef Balogh, Ramon~I Garcia, and Lina Li, \emph{Independent sets in the
  middle two layers of {B}oolean lattice}, Journal of Combinatorial Theory,
  Series A \textbf{178} (2021), 105341.

\bibitem{csikvari20}
M{\'a}rton Borb{\'e}nyi and Peter Csikvari, \emph{Matchings in regular graphs:
  minimizing the partition function}, Transactions on Combinatorics \textbf{10}
  (2021), 73--95.

\bibitem{borgs2006absence}
Christian Borgs, \emph{Absence of zeros for the chromatic polynomial on bounded
  degree graphs}, Combinatorics, Probability and Computing \textbf{15} (2006),
  63--74.

\bibitem{borgs2013left}
Christian Borgs, Jennifer Chayes, Jeff Kahn, and L{\'a}szl{\'o} Lov{\'a}sz,
  \emph{Left and right convergence of graphs with bounded degree}, Random
  Structures \& Algorithms \textbf{42} (2013), 1--28.

\bibitem{bregman1973some}
LM~Bregman, \emph{Some properties of nonnegative matrices and their
  permanents}, Soviet Math. Dokl \textbf{14} (1973), 945--949.

\bibitem{carroll2009matchings}
Teena Carroll, David Galvin, and Prasad Tetali, \emph{Matchings and independent
  sets of a fixed size in regular graphs}, Journal of Combinatorial Theory,
  Series A \textbf{116} (2009), 1219--1227.

\bibitem{cohen2020number}
Emma Cohen, Will Perkins, Michail Sarantis, and Prasad Tetali, \emph{On the
  number of independent sets in uniform, regular, linear hypergraphs}, European
  Journal of Combinatorics (to appear).

\bibitem{csikvari2016extremal}
P{\'e}ter Csikv{\'a}ri, \emph{Extremal regular graphs: the case of the infinite
  regular tree}, arXiv preprint arXiv:1612.01295 (2016).

\bibitem{CR14}
Jonathan Cutler and A.J. Radcliffe, \emph{The maximum number of complete
  subgraphs in a graph with given maximum degree}, Journal of Combinatorial
  Theory, Series B \textbf{104} (2014), 60--71 (en).

\bibitem{davies2015independentns}
Ewan Davies, Matthew Jenssen, Will Perkins, and Barnaby Roberts,
  \emph{Independent sets, matchings, and occupancy fractions}, Journal of the
  London Mathematical Society \textbf{96} (2017), 47--66.

\bibitem{DaviesCoefficients}
Ewan Davies, Matthew Jenssen, Will Perkins, and Barnaby Roberts, \emph{Tight
  bounds on the coefficients of partition functions via stability}, Journal of
  Combinatorial Theory, Series A \textbf{160} (2018), 1 -- 30.

\bibitem{friedland2008validations}
Shmuel Friedland, Elliot Krop, Per~H Lundow, and Klas Markstr{\"o}m, \emph{On
  the validations of the asymptotic matching conjectures}, Journal of
  Statistical Physics \textbf{133} (2008), 513--533.

\bibitem{friedland2008number}
Shmuel Friedland, Elliot Krop, and Klas Markstr{\"o}m, \emph{On the number of
  matchings in regular graphs}, The Electronic Journal of Combinatorics
  \textbf{15} (2008), R110.

\bibitem{galvin2014three}
David Galvin, \emph{Three tutorial lectures on entropy and counting}, arXiv
  preprint arXiv:1406.7872 (2014).

\bibitem{galvin2004weighted}
David Galvin and Prasad Tetali, \emph{On weighted graph homomorphisms}, DIMACS
  Series in Discrete Mathematics and Theoretical Computer Science \textbf{63}
  (2004), 97--104.

\bibitem{gruber1971general}
Christian Gruber and Herv{\'e} Kunz, \emph{General properties of polymer
  systems}, Communications in Mathematical Physics \textbf{22} (1971),
  133--161.

\bibitem{ilinca2013asymptotics}
Liviu Ilinca and Jeff Kahn, \emph{Asymptotics of the upper matching
  conjecture}, Journal of Combinatorial Theory, Series A \textbf{120} (2013),
  976--983.

\bibitem{jenssen2019revisited}
Matthew Jenssen and Will Perkins, \emph{Independent sets in the hypercube
  revisited}, Journal of the London Mathematical Society \textbf{102} (2020),
  645--669.

\bibitem{kahn2001entropy}
Jeff Kahn, \emph{An entropy approach to the hard-core model on bipartite
  graphs}, Combinatorics, Probability and Computing \textbf{10} (2001),
  219--237.

\bibitem{kotecky1986cluster}
Roman Koteck\'{y} and David Preiss, \emph{Cluster expansion for abstract
  polymer models}, Communications in Mathematical Physics \textbf{103} (1986),
  491--498.

\bibitem{mayer1940statistical}
JE~Mayer and MG~Mayer, \emph{Statistical mechanics}, John Wiley, 1940.

\bibitem{ordentlich2000two}
Erik Ordentlich and Ron~M Roth, \emph{Two-dimensional weight-constrained codes
  through enumeration bounds}, IEEE Transactions on Information Theory
  \textbf{46} (2000), 1292--1301.

\bibitem{penrose1967convergence}
Oliver Penrose, \emph{Convergence of fugacity expansions for classical
  systems}, Statistical Mechanics: Foundations and Applications, 1967, p.~101.

\bibitem{perarnau2018counting}
Guillem Perarnau and Will Perkins, \emph{Counting independent sets in cubic
  graphs of given girth}, Journal of Combinatorial Theory, Series B
  \textbf{133} (2018), 211--242.

\bibitem{perkins2015birthday}
Will Perkins, \emph{Birthday inequalities, repulsion, and hard spheres},
  Proceedings of the American Mathematical Society \textbf{144} (2016),
  2635--2649.

\bibitem{pulvirenti2012cluster}
Elena Pulvirenti and Dimitrios Tsagkarogiannis, \emph{Cluster expansion in the
  canonical ensemble}, Communications in Mathematical Physics \textbf{316}
  (2012), 289--306.

\bibitem{radhakrishnan1997entropy}
Jaikumar Radhakrishnan, \emph{An entropy proof of {B}regman's theorem}, Journal
  of Combinatorial Theory, Series A \textbf{77} (1997), 161--164.

\bibitem{sah2019number}
Ashwin Sah, Mehtaab Sawhney, David Stoner, and Yufei Zhao, \emph{The number of
  independent sets in an irregular graph}, Journal of Combinatorial Theory,
  Series B \textbf{138} (2019), 172--195.

\bibitem{sah2020reverse}
Ashwin Sah, Mehtaab Sawhney, David Stoner, and Yufei Zhao, \emph{A reverse
  {S}idorenko inequality}, Inventiones Mathematicae (2020), 1--47.

\bibitem{scott2005repulsive}
Alexander~D Scott and Alan~D Sokal, \emph{The repulsive lattice gas, the
  independent-set polynomial, and the {L}ov{\'a}sz local lemma}, Journal of
  Statistical Physics \textbf{118} (2005), 1151--1261.

\bibitem{sernau2018graph}
Luke Sernau, \emph{Graph operations and upper bounds on graph homomorphism
  counts}, Journal of Graph Theory \textbf{87} (2018), 149--163.

\bibitem{sokal2001bounds}
Alan~D Sokal, \emph{Bounds on the complex zeros of (di) chromatic polynomials
  and {P}otts-model partition functions}, Combinatorics, Probability and
  Computing \textbf{10} (2001), 41--77.

\bibitem{zhao2010number}
Yufei Zhao, \emph{The number of independent sets in a regular graph},
  Combinatorics, Probability and Computing \textbf{19} (2010), 315--320.

\bibitem{zhao2017extremal}
Yufei Zhao, \emph{Extremal regular graphs: independent sets and graph
  homomorphisms}, The American Mathematical Monthly \textbf{124} (2017),
  827--843.

\end{thebibliography}
\end{document}